\documentclass{amsart}

\usepackage{graphicx}
\usepackage[bottom]{footmisc}

\usepackage{tikz}
\usetikzlibrary{arrows,snakes,shapes,shadows}
\usepackage{epigraph}

\usepackage[pagebackref,bookmarksopen]{hyperref}
\hypersetup{
	colorlinks=true,
	linkcolor=blue,
	citecolor=magenta,
	urlcolor=cyan,
}

\newtheorem{theorem}{Theorem}[section]
\newtheorem{proposition}[theorem]{Proposition}
\newtheorem*{theorem*}{Theorem}

\newtheorem{lemma}[theorem]{Lemma}

\newtheorem{quest}[theorem]{Question}

\theoremstyle{definition}

\newtheorem*{amalgamation*}{Amalgamation}
\newtheorem{example}[theorem]{Example}
\newtheorem{remark}[theorem]{Remark}

\newtheorem*{remark*}{Remark}
\newtheorem{definition}[theorem]{Definition}

\newcommand\R{{\mathbb R}}
\newcommand\C{{\mathbb C}}
\newcommand\Z{{\mathbb Z}}
\newcommand\N{{\mathbb N}}
\newcommand{\norm}[1]{\lVert #1 \rVert}

\makeatletter
\def\@tocline#1#2#3#4#5#6#7{\relax
	\ifnum #1>\c@tocdepth 
	\else
	\par \addpenalty\@secpenalty\addvspace{#2}%
	\begingroup \hyphenpenalty\@M
	\@ifempty{#4}{%
		\@tempdima\csname r@tocindent\number#1\endcsname\relax
	}{%
		\@tempdima#4\relax
	}%
	\parindent\z@ \leftskip#3\relax \advance\leftskip\@tempdima\relax
	\rightskip\@pnumwidth plus4em \parfillskip-\@pnumwidth
	#5\leavevmode\hskip-\@tempdima
	\ifcase #1
	\or\or \hskip 1em \or \hskip 2em \else \hskip 3em \fi%
	#6\nobreak\relax
	\dotfill\hbox to\@pnumwidth{\@tocpagenum{#7}}\par
	\nobreak
	\endgroup
	\fi}
\makeatother
\DeclareRobustCommand{\SkipTocEntry}[5]{}

\begin{document}

\title{On Lipschitz Normally Embedded singularities}

\author{Lorenzo Fantini}
\email{\href{mailto:lorenzo.fantini@polytechnique.edu}{lorenzo.fantini@polytechnique.edu}}
\address{Centre de Mathématiques Laurent Schwartz, Ecole Polytechnique and CNRS, Institut Polytechnique de Paris,}

\author{Anne Pichon}
\email{\href{mailto:anne.pichon@univ-amu.fr}{anne.pichon@univ-amu.fr}}
\address{Aix-Marseille Universit\'e, CNRS, Centrale Marseille, I2M, Marseille, France}

\begin{abstract}
	Any subanalytic germ $(X,0) \subset ( \R^n,0)$ is  equipped with two natural metrics: its \emph{outer metric}, induced by the standard Euclidean metric of the ambient space, and its \emph{inner metric}, which is defined by measuring the shortest length of paths on the germ $(X,0)$. 
	The germs for which these two metrics are equivalent up to a bilipschitz homeomorphism, which are called \emph{Lipschitz Normally Embedded}, have attracted a lot of interest in the last decade.
	In this survey we discuss many general facts about Lipschitz Normally Embedded singularities, before moving our focus to some recent developments on criteria, examples, and properties of Lipschitz Normally Embedded complex surfaces. 
	We conclude the manuscript with a list of open questions which we believe to be worth of interest.
\end{abstract}

\maketitle

\bigskip

\hfill
\begin{minipage}[t]{0.59\textwidth}
\textit{This paper is dedicated to Walter Neumann,\\ wonderful friend and outstanding mathematician.}
\end{minipage}

\tableofcontents

\section{Definition, first examples, and some general results}

\label{sec:1}

\subsection{Definition}

Let $(X,d_X)$ and $(Y,d_Y)$ be two metric spaces.
A homeomorphism $\phi \colon X \to Y$ is said to be a \emph{bilipschitz equivalence} if there exist two positive real numbers $K_1$ and $K_2$ such that, given any two points $x$ and $x'$ in $X$, we have
\[
K_1\  d_X(x, x') \leq  d_Y\big(\phi(x), \phi(x')\big)  \leq K_2 \  d_X(x, x').
\]
Two metric spaces are said to be \emph{bilipschitz equivalent} if there exists a bilipschitz equivalence from one to the other.

A subanalytic subspace $X$ of $\R^n$ is naturally equipped with two metrics on $(X,0)$: its \emph{outer metric} $d_{o}$, induced by the standard Euclidean metric of the ambient space, and its \emph{inner metric} $d_{i}$, which is the associated arc-length metric on the germ, defined as follows: 
\[
d_i(x,y) = \inf\big\{\mathop{\rm length}(\gamma)\,\big\vert\,
\gamma\text{ is a rectifiable path in }X\text{ from }x\text{ to
}y\big\}.
\]
Note that for an arc to be \emph{rectifiable} essentially means that its length can be computed and is finite, see \cite{Gromov1981} for details. Given any two points $x$ and $y$ in $X$, we have $d_o(x,y) \leq d_i(x,y)$.
Moreover, the inner distance between two given points can be computed as a limit of sums of outer distances, so that two spaces which are bilipschitz equivalent for the outer metric are bilipschitz equivalent for the inner metric as well. 
In general, the converse  does not hold, but there exists a special class of spaces, or of space germs, which have the remarkable property that their inner and outer bilipschitz classes coincide, in the following sense.

\begin{definition} 
	A subanalytic subspace $X$ of $\R^n$ is \emph{Lipschitz Normally Embedded} (or simply \emph{LNE}) if  the identity map on $X$ is a bilipschitz equivalence between its inner and outer metrics, that is if there exists a real number $K\geq 1$ such that, for all $x,y$ in $X$, we have
\[
\frac{1}{K} d_i(x,y) \leq d_o(x,y).
\]
	If $x$ is a point of $X$, the germ $(X,x)$ is \emph{LNE} if there is a neighborhood $U$ of $x$ in $\R^n$ such that $X \cap U$ is LNE.
\end{definition}


Since the inner and the outer geometries of $(X,x)$ are invariant under bilipschitz homeomorphisms (see \cite[Proposition~7.2.13]{Pichon2020}), this property only depends on the subanalytic type $(X,x)$, and not on the choice of an embedding in some smooth ambient space $(\R^n,0)$.\footnote{Note that while the result of \emph{loc. cit.} is only stated in the semialgebraic setting, its proof carries through for arbitrary subanalytic germs.}

This definition was first introduced by Birbrair and Mostowski in \cite{BirbrairMostowski2000}; the main result of that seminal paper is presented in Subsection~\ref{subsection:normal_reembedding_and_pancake}.
Note that in \emph{loc. cit.} LNE spaces are simply called {\it normally embedded}; in the subsequent literature on the subject the term {\it Lipschitz} was added to distinguish this notion from those of projective normal embedding (in algebraic geometry) and normality (in local geometry, commutative algebra and singularity theory).

Notice that a compact space $X$ is LNE if and only if the germs $(X,x)$ are LNE for all points $x$ if $X$. 
Our aim is to present a state of the art on the LNE-ness of real and complex analytic germs.

\subsection{First examples}

\begin{example}
	A smooth germ $(X,0)$ is Lipschitz Normally Embedded, since it is analytically equivalent to $(\R^n,0)$, where the inner metric and the outer metric coincide. 
\end{example}

\begin{example} 
	Let $Y \subset \R^n$ be a subanalytic subspace of the sphere $S^{n-1}$ of radius $1$ centered at the origin of $\R^n$, and assume that $Y$ is LNE.
	Then the cone $C(Y,0)$ over $Y$ with apex $0$, which consists of the union of the half-lines with origin $0$ that intersect $Y$, is LNE as well. 
\end{example}

\begin{example} \label{ex:non LNE}
	The germ $(C,0)$ of the real cusp $C$ with equation $y^2-x^3=0$ in $\R^2$ is not LNE.
	Indeed, given a real number $t>0$, consider the two points $p_1(t) = (t,t^{3/2})$ and $p_2(t) = (t,-t^{3/2})$ on $C$ (see Figure~\ref{fig:1}). 
	Then $d_o\big(p_1(t), p_2(t)\big) = 2 t^{3/2}$, so that in the germ, as $t$ goes two zero, the outer distance between $p_1(t)$ and $p_2(t)$ has order $t^{3/2}$, which we write as $d_o\big(p_1(t), p_2(t)\big) = \Theta(t^{3/2})$.%
	\footnote{More precisely, throughout this text, we use the \emph{big-Theta} asymptotic notations of Bachmann--Landau in the following form: given two function germs $f,g\colon ([0,\infty),0)\to ([0,\infty),0)$, we say that $f$ \emph{is big-Theta of} $g$, and we write $f(t) = \Theta (g(t))$, if there exist $\eta>0$ and $K >0$ such that ${K}^{-1}g(t) \leq f(t) \leq K g(t)$ for all $t$ satisfying $f(t)\leq \eta$.}
	On the other hand, the  shortest path on $C$ between the two points $p_1(t)$ and $p_2(t)$ is obtained by taking a path going through the origin, so that we have $d_i\big(p_1(t), p_2(t)\big) = \Theta(t)$. 
	Therefore, taking the limit of the quotient as $t$ tends to $0$, we obtain: 
	\[
	\frac{d_o\big(p_1(t), p_2(t)\big)}{d_i\big(p_1(t), p_2(t)\big)} = \Theta(t^{1/2}) \longrightarrow 0.
	\]
	Note that the existence of two such arcs $p_1$ and $p_2$ is due to the fact that the tangent cone $T_0X$ of $(X,0)$ at $0$ is not reduced (it has equation $y^2=0$). 
	This is an occurrence of a general result which will be stated as Theorem \ref{thm:tangent cone}.
    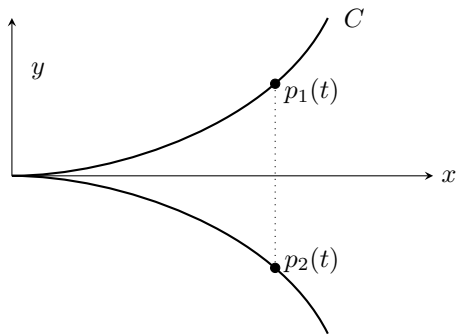
\begin{figure}[ht]
    \centering
\begin{tikzpicture} 
\begin{scope}[scale=0.7]
\draw[thick](0,0) .. controls (2,0) and (5,1)..  (6,3) ;
\draw[thick](0,0) .. controls (2,0) and (5,-1)..  (6,-3) ;

\draw[thin,>=stealth,->](0,0)--(8,0);
\draw[thin,>=stealth,->](0,0)--(0,3);

\draw[fill](5,1.75)circle(2.5pt);
\draw[fill](5,-1.75)circle(2.5pt);

\draw[dotted](5,1.75)--(5,-1.75);
    
    \node(a)at(8.3,0){$x$};
    \node(b)at(6.5,3){$C$};
     \node(b)at(5.7,1.6){$p_1(t)$};
          \node(b)at(5.7,-1.6){$p_2(t)$};

              \node(c)at(0.5,2){$y$};
                                 
     \end{scope}             
\end{tikzpicture}
  \caption{The real cusp $y^2-x^3=0$}
    \label{fig:1}
\end{figure}
\end{example}

\begin{example} 
	A complex curve germ $(C, 0) \subset (\C^N,0)$ is LNE if and only if it consists of smooth transversal curve germs.
	Indeed, if the latter is true then $(C,0)$ is analytically equivalent to the germ of a union of transversal lines, which being a cone is LNE.
	The converse is more delicate and can be obtained by combining several results. 
	First, if  $(C, 0) \subset (\C^N,0)$ is a complex curve germ, then any \emph{generic} linear projection $\ell \colon \C^N \to \C^2$  restricts to the germ of a bilipschitz homeomorphism $\ell |_{(C,0)} \colon (C,0) \to \big(\ell(C), 0\big)$ for the outer metric  (\cite[pp. 352-354]{Teissier1982}). 
	Therefore, it suffices to prove the result for a plane curve  $(C, 0) \subset (\C^2,0)$.  
	The key argument, which is close to the one presented in Example~\ref{ex:non LNE}, is that a complex curve germ $(C,0) \subset (\C^2,0)$ admitting a non essential Puiseux exponent $q>1$ contains two arcs $p_1(t)$ and $p_2(t)$ such that $d_i\big(p_1(t), p_2(t)\big) = \Theta(t)$ and  $d_o(p_1(t), p_2(t)) = \Theta(t^q)$, and therefore $(C,0)$  cannot be LNE. 
	For example, the complex cusp $y^2-x^3=0$ in $\C^2$, which has has Puiseux expansion $y=x^{3/2}$, contains the two  arcs $p_1(t) =  (t,t^{3/2})$ and  $p_2(t) = (t,- t^{3/2})$ whose inner distance is $\Theta(t)$ and whose outer distance is $\Theta(t^{3/2})$. 

	Notice that, more generally, the inner and outer  bilipschitz types of complex curve germs are completely understood. 
	On the one hand, the inner bilipschitz geometry of a complex curve $(C, 0)$ is trivial in the sense that for the inner metric $(C, 0)$ is bilipschitz equivalent to a straight cone over its link, that is to a union of smooth transversal curve germs (see \cite[Proposition 7.2.2]{Pichon2020}).  
	On the other hand, the outer bilipschitz type of $(C,0)$ determines and is determined by its embedded topological type; for an algebraic proof  of this result involving Lipschitz saturation of ideals, see the pioneering paper by Pham and Teissier \cite{PhamTeissier1969} or its recent English translation \cite{PhamTeissier2020}; for a more geometric approach, see \cite{Fernandes2003} or \cite{NeumannPichon2014}. 
\end{example}

\begin{example} 
	Starting from dimension 3 it is easy to find examples of non-LNE complex analytic germs which have non-isolated singularities, for example by taking the product of a non-LNE germ with a line.
	For instance, the product of a real cusp with a real line, that is the complex hypersurface in $\C^3$ with equation $y^2-x^3=0$, is not LNE. 
	One gets other examples by taking a homogeneous complex space with a non-LNE link; such an example is given by the hypersurface germ in $(\C^3,0)$ with equation $x^2z+y^3=0$
\end{example}

\begin{example} 
	The first examples of non-LNE complex surface germs with an isolated singularity were obtained by Birbrair, Fernandes, and Neumann in 2010 \cite{BirbrairFernandesNeumann2010}. 
	It is the family of  Brieskorn surfaces $x^b + y^b + z^a = 0$  where $b>a$ and $a$ is not a divisor of $b$. 
	In fact, while few examples of families of LNE singularities are known, it is still unclear whether LNE-ness is common among complex singularities with isolated singularities, even in the case of surfaces.  
	The second part of the present paper discusses several recent advances on this front.   
\end{example}

\begin{example} 
	The space of $n \times m$ real and complex matrices also contain remarkable families of LNE subspaces. 
	For example, the Lie group $GL^+_n(\R)$ consisting of $n \times n$ matrices with positive determinant is LNE, and so are the set of $n \times n$ matrices $X^{n-1}$ with rank $n-1$ and its closure, which is the set of matrices of determinant zero (\cite{KatzKatzKernerLiokumovich2018}).
	These results are generalized in \cite{KernerPedersenRuas2018} to the sets $X_t$ of $m \times n$ matrices of given rank $t \leq \min(m,n)$ and their closures $\overline{X_t}$ by using elementary arguments of linear algebra and trigonometry, and LNE-ness is also proved in \emph{loc.\ cit.} for other families such as symmetric and skew-symmetric matrices of given rank $t$ and their closures, upper triangular matrices with determinant zero, and the intersections of those spaces with some linear subspaces.  
\end{example}

\subsection{The pancake decomposition, the pancake metric, and the embedding problem}
\label{subsection:normal_reembedding_and_pancake}

In this subsection, we present three important theorems which can be considered as the first historical results around Lipschitz Normal Embeddings.
We will state them in the semialgebraic setting, but they remain true in the subanalytic and polynomially bounded $o$-minimal categories with the obvious adaptations.

Since LNE spaces can be thought of as the simplest ones with respect to inner and outer Lipschitz geometries, it is natural to ask whether every semialgebraic subset of $\R^n$ admits a finite decomposition as a union of LNE sets.
The answer is positive, as was established by Parusinski and Kurdyka:

\begin{theorem} [Pancake Decomposition \cite{Parusinski1988,Kurdyka1992,Parusinski1994}] \label{theorem:pancake decomposition} 
Let $X \subset \R^n$ be a closed
semialgebraic set. 
Then we can write
\[
X=\bigcup_{i=1}^r  X_i
\]
as a finite union of closed semialgebraic subsets of such that:   
\begin{enumerate}
\item  all $X_i$ are LNE;
\item  for every $i \neq j$ we have $\dim(X_i \cap X_j)< \min( \dim X_i, \dim X_j) $.
\end{enumerate}
\end{theorem}

This remarkable result has several important consequences. 
First, it enables to approach the following natural question: given a closed connected  subset semialgebraic subset $X$ in $\R^n$ is the inner metric $d_i\colon X\times X\to \R_{\geq 0}$ a semialgebraic function?
Note that this is clearly the case for the outer metric on $X$.

The following theorem, proved by Kurdyka and Orro, states that $d_i$ is bilipschitz equivalent to a semialgebraic metric with a bilipschitz constant as close as we want from $1$. 
To define such a semialgebraic metric, consider a pancake decomposition $P = \{X_i\}_{i=1}^r$ of $X$. 
Given two points $x, y $ in $X$ let $Z_{x,y}$ be the set consisting of all the finite ordered sequences $z = (z_1,\ldots,z_s)$ of points  on $X$ such that 
$z_1=x$, $z_s=y$,  and for every $k\in\{1,\ldots, s-1\}$, there is a pancake 
$X_{i_k}$ such that $X_{i_k} \cap \{ z_1,\ldots,z_s\} = \{z_k , z_{k+1}\}$. 
Finally, set 
\[
d_P(x,y) = \inf_{(z_1,\ldots,z_s) \in Z_{x,y}}   \sum_{k=1}^{s-1} d_i(z_k, z_{k+1}).
\]

\begin{theorem}[Pancake metric, \cite{KurdykaOrro1997}]
	\label{theorem:pancake_metric}
The function $d_p \colon X \times X \to \R$ is semialgebraic and defines a metric on $X$ (called the \emph{pancake metric}) which is bilipschitz equivalent to $d_i$.
Moreover, for all $\epsilon >0$, there exists a pancake decomposition (obtained by refinement), such that the underlying pancake metric satisfies
$$\forall x, y \in X, \ d_i(x,y) \leq d_P(x,y)  \leq (1+\epsilon)  d_i(x,y). $$
\end{theorem}


An important application of the result above is the solution by Birbrair and Mostowski of the embedding problem, which asks whether every compact connected semialgebraic set is inner bilipschitz equivalent to a LNE semialgebraic set:
 
\begin{theorem}[\cite{BirbrairMostowski2000}]
	\label{theorem:embedding_problem}
Let $X$ be a compact connected semialgebraic subset of $\R^n$. 
Then, for every $\epsilon >0$, there exists a semialgebraic set $X_{\epsilon} \subset \R^m$ such that:
\begin{enumerate}
\item  $X_{\epsilon} $ is semialgebraically bilipschitz equivalent to $X$ with respect to the inner metric;
\item $X_{\epsilon} $ is LNE; 
\item  the Hausdorff distance between $X$ and $X_{\epsilon}$  is less than $\epsilon$.
\end{enumerate}
\end{theorem}

\subsection{Characterization of LNE-ness via arcs} \label{subsec:arc criterion}

 we recall a  necessary and sufficient condition for the LNE-ness of a semialgebraic set which was proved by Birbrair and Mendes.
As in the previous subsection, the results stay true in the subanalytic or more generally polynomially bounded $o$-minimal setting (see \cite[Remark~2.3]{BirbrairMendes2018}).

\begin{definition} 
	Let $(X,0) \subset (\R^n,0)$ be a semialgebraic germ.  
	A \emph{real arc} on $(X,0)$ is the germ of a semialgebraic map $\delta \colon [0,\eta) \to X$ for some $\eta >0$, such that $\delta(0)=0$ and $\norm{\delta(t)}=t$ (see also Remark \ref{rk:parametrization}).\\
	When no risk confusion may arise, we will use the same notation for a real arc $\delta$ and for the germ $\big(\delta([0,\eta)),\delta(0)\big)$ of its parametrized image. 
\end{definition}

\begin{definition} \label{definition:inner_and_outer_contacts}
	Let $(X,0) \subset (\R^n,0)$ be a semialgebraic germ and let $\delta_1 \colon [0,\eta)  \to X$ and $\delta_2  \colon [0,\eta)  \to X$  be two real arcs on $X$. 
	The {\it outer contact} of $\delta_1$ and $\delta_2$ is defined to be infinity if $\delta_1=\delta_2$ and is otherwise the rational number $q_{o}=q_{o}(\delta_1,\delta_2)$ defined by
\[
\norm{\delta_1(t)-\delta_2(t)} = \Theta(t^{q_{o}}).
\]
	The {\it inner contact} of $\delta_1$ and $\delta_2$  is the rational number $q_{i}=q_{i}(\delta_1,\delta_2)$ defined by
\[
d_{i}\big(\delta_1(t),\delta_2(t)\big) = \Theta(t^{q_i}).
\]
\end{definition}

\begin{remark} 
	The existence and rationality of the inner contacts $q_i$ is a consequence of the fact that the inner metric is bilipschitz equivalent to the pancake metric (Theorem~\ref{theorem:pancake_metric}), which is semialgebraic.
\end{remark}

\begin{remark}\label{rk:parametrization} 
	The inner and outer contacts $q_{i}(\delta_1,\delta_2)$ and $q_{o}(\delta_1,\delta_2)$ can also be defined taking reparametrizations by \emph{real slices} of $\delta_1$ and $\delta_2$ as follows. 
	First note that if $\delta_1$ and $\delta_2$ have different tangent directions then $q_i(\delta_1,\delta_2)=q_o(\delta_1,\delta_2)=1$, so we may assume that they have the same tangent direction. 
	We can then choose coordinates $(x_1,\ldots,x_n)$ such that along the tangent half-line of $\delta_1$ and $\delta_2$ we have $x_1>0$ except at $0$. 
	For $j=1,2$, consider the reparametrization $\tilde{\delta}_j \colon [0,\eta) \to \R^n$ defined by $\tilde{\delta}_j(t) = \delta_j \cap \{x_1=t\}$. 
	Then we have $\norm{\tilde{\delta}_1(t)-\tilde{\delta}_2(t)} = \Theta(t^{{q_o}}) $ and $d_{i}\big(\tilde{\delta}_1(t),\tilde{\delta}_2(t)\big) = \Theta(t^{{q_i}})$.
	Indeed, this is an easy consequence of the following standard lemma:
	\begin{lemma}
		Let $B\subset \R^n$ be any closed compact convex neighborhood of $0$ in $\R^n$ and denote by $B_1$ is the unit ball of $\R^n$. 
		Let $\phi\colon B\to B_1$ be the homeomorphism which maps each ray
		from $0$ to $\partial B$ linearly to the ray with the same tangent, but of length $1$. 
		Then the map $\phi\colon B\to B_1$ is a bilipschitz homeomorphism.
	\end{lemma}
\end{remark}

We can now state the main result of this subsection, which is a criterion to determine if a closed semialgebraic germ is LNE using arcs and their contact orders.

\begin{theorem}[Arc criterion, \cite{BirbrairMendes2018}] 
		\label{theorem:criterion_arcs_original}
		\label{thm:arc criterion}
	Let $(X,0)\subset (\R^n,0)$ be a closed semialgebraic germ. 
	Then $(X,0)$ is LNE if and only if all pairs of real arcs $\delta_1$ and $\delta_2$ in $(X,0)$ satisfy $q_i(\delta_1,\delta_2) = q_o(\delta_1,\delta_2)$.
\end{theorem}  

The proof of this theorem is based on the Curve Selection Lemma.
Since the latter only applies to semialgebraic metrics, the semialgebraicity of the pancake metric and Theorem~\ref{theorem:pancake_metric} play again a fundamental role.

\begin{example} \label{example:real surface} 
	A straightforward application of the arc criterion shows that the real surface $S$ in $\R^3$ defined by the equation $x^2+y^2-z^3=0$ is LNE. 
	See also Example~\ref{example:real surface2}.
\end{example}

The criterion given in Theorem~\ref{theorem:criterion_arcs_original} is difficult to use effectively in practice since it requires to compute the inner and outer contact orders of an immense amount of pairs of arcs.
In Section~\ref{sec:surfaces} we state an analogous criterion for complex surface germs where the number of pairs of arcs to be tested is reduced drastically to just finitely many pairs. 
This makes the criterion much more efficient to prove LNE-ness and enables one to obtain several infinite families of LNE complex surface germs with isolated singularities.

\subsection{Characterization of LNE-ness via the links}

Recall that the \emph{link} of a $d$-dimensional subanalytic germ $(X,0)\subset \R^n$, which is defined by embedding $(X,0)$ in a suitable smooth germ $(\mathbb{C}^N,0)$ and intersecting it with a small sphere, is, up to homeomorphism, a well defined real $(2d-1)$-dimensional oriented pseudo-manifold (a smooth manifold if $(X,0)$ has isolated singularities) which determines and is determined by the homeomorphism class of the germ $(X,0)$.
In this subsection we discuss the relation between a germ being LNE and its link being LNE.
One implication is always satisfied:

\begin{lemma} \label{lemma:LNE_implies_link_is_LNE}
	Let $(X,0)$ be a subanalytic germ in $\R^n$ such that $(X \setminus\{0\},0)$ is connected. 
	Then, if $(X,0)$ is LNE, so is its link. 
\end{lemma}

This is a consequence of the fact that, whenever the link of $(X,0)$ is connected, given two real arcs $\delta_1$ and $\delta_2$ as in Definition~\ref{definition:inner_and_outer_contacts}, their inner contact can be computed as the asymptotic of the inner distances between the points $\delta_1(t)$ and $\delta_2(t)$ on the representative $X\cap \{||x||=t\}$ of the link of $(X,0)$.
The converse implication is only true in some special cases, such as 
for conical subset of $\R^n$, as treated by Kerner, Pedersen, and Ruas:

\begin{proposition}[{\cite[Proposition 2.8]{KernerPedersenRuas2018}}]  \label{prop:cone} 
	Let $S^{n-1}$ be the unit sphere centered at the origin of $\R^n$, let $M$ be a compact subset of  $S^{n-1}$, and let $X= C(M) \subset \R^n$ be the cone over $M$, that is the union of the half-lines with origin $0$ and passing through points of $M$. 
	Then $X$ is LNE if and only if $M$ is LNE (as a subset of $\R^n$).
\end{proposition}

The proof is obtained by performing direct computations of inner and outer distance between points inside $C(M)$.  

\begin{remark}
In the case where $M$ does not intersect the meridian sphere $S^{n-2} = S^{n-1} \cap  \R^{n-1}\times \{0\} \subset \R^n$, then $X= C(M)$ is also the cone $C(N)$ over the compact set $N = C(M) \cap \R^{n-1}\times \{\pm1\}$, and the map $\phi \colon M \to N$ sending a point $x$ of $M$ to the point of the half-line through $x$ which intersects $\R^{n-1}\times \{\pm1\}$ realizes a bilipschitz homeomorphism for the outer metric. 
Therefore, the LNE-ness of $M$ is equivalent to that of $N$ in this case.
\end{remark}

\begin{example}
	Consider the cone $C(N)$ in $\R^3$ over the union of the two circles
\[
N = \big\{(x,y,1) \in \R^3 \,\big|\, \big((x-1)^2+y^2-1\big) \big((x+1)^2+y^2-1\big)=0  \big\}.
\] 
	Then $C(N)$ is not LNE since $N$ is not LNE at the intersection point $q=(0,0,1)$ of the two circles. 
	Indeed, the two arcs $p_1(t) = (-1+\sqrt{1-t^2})$ and $p_2(t) = (1-\sqrt{1-t^2})$ on $(N,q)$ satisfy $q_o(p_1,p_2)={3}/{2}\neq q_i(p_1,p_2) = 1$.
\end{example}

In \cite{MendesSampaio2021}, Mendes and Sampaio proved a broad generalization of Proposition~\ref{prop:cone} which provides a characterization of LNE subanalytic germs via their links. 
This result was further generalized by Nguyen in \cite{Nguyen2021} to any definable set in a $o$-minimal structure (not necessarily polynomially bounded). 
We state here this most general version.

\begin{theorem} \label{thm:LLNE} \cite{MendesSampaio2021, Nguyen2021}   
	Let $(X,0)$ be a definable germ in $(\R^n,0)$ and let $\rho \colon (X,0) \to (\R,0)$ be the germ of a continuous definable function such that $\rho(x)=\Theta(\norm{x})$. 
	Suppose that $(X \setminus\{0\},0)$ is connected. 
	Then the following statements are equivalent:
\begin{enumerate}
	\item $(X,0)$ is LNE;
	\item \label{eq:LLNE} There exist real numbers $r_0>0$ and $C>0$ such that, for every $r\in (0,r_0]$, the set $X_r = \rho^{-1}(r) \cap X$ is LNE with Lipschitz constant bounded by $C$. 
\end{enumerate}
\end{theorem}

\begin{remark}
	Condition \eqref{eq:LLNE} is stated in \cite{MendesSampaio2021} in the case where the function $\rho$ equals the distance to the origin.
	In that case, $X_r = S^{n-1}_r \cap X$ is the link of $(X,0)$ at distance $r$ and a germ $(X,0)$ satisfying condition~\eqref{eq:LLNE} is said to be \emph{link-LNE} (or simply \emph{LLNE}).
\end{remark}

The proof of Theorem~\ref{thm:LLNE} in the case where $(X,0)$ and $\rho$ are subanalytic is based on the Curve Selection Lemma, used in a similar way as in the proof of Theorem \ref{theorem:criterion_arcs_original} given in \cite{BirbrairMendes2018}, and on a result of Valette \cite[Corollary 2.2]{Valette2007} which states the existence of a bilipschitz homeomorphism $h \colon (X,0) \to (X,0)$ such that for all $x$ in a neighborhood of the origin we have $\norm{h(x)} = \rho(x)$.

\begin{example} \label{example:real surface2} 
	As an application of Theorem \ref{thm:LLNE},  consider again the real hypersurface $S$ defined in $\R^3$ by the equation $x^2+y^2-z^3=0$ of Example~\ref{example:real surface2}, and fix $t>0$. 
	Then the intersection $S_t = S \cap\{z=t\}$ is a circle with radius $t^{3/2}$. 
	Therefore, every $S_t$ is LNE with Lipschitz constant $C= 2\pi$, so that $S$ is link-LNE and hence LNE.    
\end{example}  

\begin{example} \label{ex:Nguyen} \cite[Proposition 3.11]{Nguyen2021} Consider the semialgebraic germ $(X,0)$ in $(\R^3,0)$ defined by $X=\{(t,x,z) \in \R^3 | 0 \leq x \leq t, z^2=t^2x^2  \}$ and the semialgebraic function  $\rho \colon (X,0) \to (\R^+,0)$ define by $(t,x,z) \mapsto t$. 
Then $\rho(w) = \Theta(\norm{w})$ and $X_r = \rho^{-1}(r) \cap X$ is LNE but its Lipschitz constant is $\Theta({1}/{r})$. 
Therefore, condition \eqref{eq:LLNE} of Theorem~\ref{thm:LLNE} is not satisfied, which implies that $(X,0)$ is not LNE.
\end{example}

\subsection{LNE-ness and Moderately Discontinuous homology}

In \cite{BobadillaHeinzePereiraSampaio2019}, Fern\'andez de Bobadilla, Heinze, Pe Pereira, and Sampaio defined  a homology theory called \emph{Moderately Discontinuous homology}.
It produces families of groups which are invariants of the bilipschitz homeomorphism classes of subanalytic germs with respect to either the inner or the outer metric. 
In particular, given a subanalytic germ $(X,0)\subset(\R^n,0)$, the identity map on $(X,0)$ induces homomorphisms between the corresponding Moderately Discontinuous homology groups of $(X,0)$ with respect to these two metrics, and it is easy to check that if $(X,0)$ is LNE then these homomorphisms are isomorphic. 
The authors asked whether the converse is true: 

\begin{quest}
	Let $(X,0)\subset(\R^n,0)$ be a subanalytic germ and assume that the identity map induces isomorphisms at the level of Moderately Discontinuous homology with respect to the inner and the outer metric at every point of $(X,0)$.
	Is $(X,0)$ necessarily LNE?
\end{quest}

In general, the answer is no: Example~\ref{ex:Nguyen} is a counter-example, as shown by Nguyen in \cite[Proposition 3.11]{Nguyen2021}.  
Notice that that example is semialgebraic and has a non-isolated singularity. 
The question is still open in the case of an isolated singularity or in the complex analytic setting.

\subsection{LNE-ness and tangent cones}

In this subsection, we state and discuss two necessary conditions for the LNE-ness of a subanalytic germ  $(X,0) \subset (\R^n,0)$ in term of its tangent cone, proved by Fernandes and Sampaio.

\begin{theorem}[{\cite[Corollary 3.11]{FernandesSampaio2019}}] 
		\label{thm:tangent cone} 
	Let $(X,0) \subset (\R^n,0)$ be a subanalytic germ and let $T_0X$ be its tangent cone at $0$. 
	If $X$ is LNE, then the two following conditions are satisfied:
\begin{enumerate}
\item \label{eq:tangent cone LNE} $T_0X$ is LNE;
\item  \label{eq:tangent cone reduced} $T_0X$ is reduced. 
\end{enumerate}
\end{theorem}

\begin{proof}[Sketch]
	The proof of the first part uses the following notion of tangent cone of a subanalytic germ in $(X,0) \subset (\R^n,0)$, introduced in \cite[Section 2.2]{FernandesSampaio2019}, which generalizes the classical definition in the real or complex analytic setting.
	Let $ D_0(X)$ be the set of unitary vectors $v$ in $\R^n \setminus \{0\}$ such that there exist a sequence of points $(x_j)_{j \in \N}$ in $X \setminus \{0\}$ converging to $0$ such that $\lim_{j \to +\infty} \frac{x_j}{\norm{x_j}}=v$;  the {\it tangent cone} $T_0X$ of $(X,0)$ at $0$ is defined by 
\[
	T_0X = \{tv \  |  \ v  \in D_0(X), t \in \R^+\}.
\]
	Assume that $(X,0)$ is LNE. 
	Let $0 \in U \subset X$ be a small neighborhood of $0$ in $X$ and let  $\lambda >0$ such that for all $x,y \in U, d_i(x,y)\leq \lambda d_o(x,y)$. 
	The proof of the first part of the theorem presented in \cite{FernandesSampaio2019} considers two vectors $v,w$ in $T_0X$ and constructs an arc $\alpha$ in $T_0X$ between $v$ and $w$  with length at most $(1+\lambda)\norm{x-y}$. 
	The arc $\alpha$ is obtained by an elegant argument using the Arzelà--Ascoli Theorem, as the limit of arcs joining two sequences of points  $(x_j)$ and $(y_j)$ in $(X,0)$ such that $\lim  \frac{x_j}{\norm{x_j}}=v$ and $\lim  \frac{y_j}{\norm{y_j}}=w$.  
	The second part of the theorem requires a definition of {\it reducedness} for the tangent cone, introduced in \cite{BirbrairFernandesGrandjean2017} and based on an equivalent definition of the tangent cone $T_0X$ using spherical blowups. 
	We refer to \cite[Section 3]{FernandesSampaio2019} for details, and only remark that the definition coincides with the classical one in the case of an analytic germ. 
	The proof consists then in the construction of a pair of arcs $(p_1, p_2)$ which does not satisfy the arc criterion in the neighborhood of a non reduced component of $T_0X$, in a similar way as in Example~\ref{ex:non LNE}. 
\end{proof}

The converse of Theorem~\ref{thm:tangent cone} is not true.
It is easy to find counter-examples among semialgebraic germs with non-isolated singularities, such as $X= \big\{(x,y,z)\in \R^3 \,\big|\, \big(x^2+y^2-z^2\big)\big( x^2+(|y| -z-z^3)^2-z^6\big)=0, z \geq 0\big\}$ (\cite[Example 3.12]{FernandesSampaio2019}). 

Note that in the example above the link of $(X,0)$ is not LNE, hence $(X,0)$ cannot be LNE itself thanks to Lemma~\ref{lemma:LNE_implies_link_is_LNE}.
Therefore, it becomes natural to ask the following question: given a subanalytic germ $(X,0) \subset (\R^n,0)$ whose link is LNE and satisfying Conditions~\eqref{eq:tangent cone LNE} and \eqref{eq:tangent cone reduced} of Theorem~\ref{thm:tangent cone}, is $(X,0)$ necessarily LNE? 
The answer is negative, even among complex germs with isolated singularities. 
A counter-example is given by Neumann and the second author in the appendix of \cite{FernandesSampaio2019}: \footnote{Note that the link of a subanalytic germ with isolated singularities is smooth, and therefore LNE.}

\begin{proposition} 
	The hypersurface germ  in $(\C^3,0)$ with equation
\[
	y^4+z^4+x^2(y+2z)(y+3z)^2+(x+y+z)^{11}=0
\]
	is not LNE , it has an isolated singularity at $0$, and its tangent cone is reduced and LNE. 
\end{proposition}

To end this subsection, let us mention that Fernandes and Sampaio proved in \cite{FernandesSampaio2020}  the following analogue of Theorem~\ref{thm:tangent cone} about complex algebraic sets of any dimension which are LNE at infinity, recovering in particular the results of \cite{DiasRibeiro2021}.

\begin{theorem} 
	Let $X$ be  complex analytic set in $\C^n$. 
	Assume that:
\begin{enumerate}
\item $X$ is Lipschitz Normally Embedded at infinity, that is there exists a compact subset $K$ of $\C^n$ such that each connected component of $X \setminus K$ is Lipschitz Normally Embedded;
\item  The tangent cone of $X$ at infinity is a linear subspace of $\C^n$.
\end{enumerate}
Then $X$ is an affine linear subspace of $\C^n$.
\end{theorem}

We refer to \cite{FernandesSampaio2020} for details, and in particular to \cite[Section 4]{FernandesSampaio2020} for the definition of the tangent cone at infinity.
This result is remarkable since it shows that an a priori mild assumption at infinity forces the rigidity of the whole $X$.
Note that if $X$ is Lipschitz regular at infinity, that is if outside of a large compact set in $\R^n$ it is bilipschitz homeomorphic to an open subset of $\R^k$ for some $k$, then $X$ is LNE at infinity (this is \cite[Corollary 3.3]{FernandesSampaio2020}).

\subsection{LNE in topology and other fields}

Lipschitz Normal Embeddings have also been useful to study problems in topology.
For example, Birbrair, Mendes, and Nu\~no-Ballesteros \cite{BirbrairMendesNuno-Ballesteros2018} prove that for a large class of real analytic parametrized surfaces in $\R^4$ LNE-ness implies the triviality of the knot obtained as their link.

More recently, Fernandes and Sampaio \cite[Theorem~3.2]{FernandesSampaio2021} showed that two LNE compact subanalytic sets which are close enough with respect to the Hausdorff distance have isomorphic fundamental groups.
This leads them to give topological conditions on the link of a LNE germ that ensure that the germ is smooth, (see Theorem~4.1 in \emph{loc. cit.}), from which they derive the following remarkable result, which is a metric version, in arbitrary dimension, of Mumford's link criterion for the smoothness of normal surface germs.

\begin{theorem}[{\cite[Theorem~4.2]{FernandesSampaio2021}}]
	Let $(X,0)$ be a complex analytic germ of dimension $k$ with isolated singularities. 
	Then $(X,0)$ is smooth if and only if it is locally metrically conical and its link at $0$ is $(2k-2)$-connected.
\end{theorem}

The definition for $X$ being locally metrically conical at $0$ is given in \cite[page 4]{FernandesSampaio2021}.
We note that an earlier result towards a metric characterization of smoothness was obtained by Birbrair, Fernandes, Lê, and Sampaio \cite{BirbrairFernandesLeSampaio2016}, who proved that a germ which is outer bilipschitz equivalent to a smooth germ $(\C^m,0)$ is itself smooth.

We also remark that a problem related to the embedding problem discussed in Subsection~\ref{subsection:normal_reembedding_and_pancake} is studied in functional analysis.
Indeed, some people working in that field are interested in studying different classes of embeddings (some of which closely resemble those of the Theorem~\ref{theorem:embedding_problem} of Birbrair and Mostowski) of \emph{discrete} metric spaces into suitable Banach spaces.
We refer the interested reader to the monograph~\cite{Ostrovskii2013} and to the many references found therein.


\section{LNE among complex surface germs}
\label{sec:surfaces}

The goal of this section is to overview some recent advances on the study of LNE singularities among complex surface germs.

\subsection{A refinement of the arc criterion}

As was mentioned in subsection~\ref{subsec:arc criterion}, the arc-based criterion for LNE-ness of Birbrair and Mendes given in Theorem~\ref{theorem:criterion_arcs_original} is difficult to use effectively in practice, as it requires to compute inner and outer contact orders of infinitely many pairs of arcs.
Whenever $(X,0)$ is a normal surface germ, this situation has been improved upon by Neumann, Pedersen, and the second author of this survey (see \cite{NeumannPedersenPichon2020a}), and then further in an upcoming work by Pedersen, Schober, and the two authors (see \cite{FantiniPedersenPichonSchober2022}), leading to a drastic reduction of the amount of pairs of real arcs whose contact orders have to be compared, down to a finite (and in fact rather small) number.


In order to state the improved criterion we need to introduce the notion of test curve.
Given a sequence of point blowups $\rho\colon Y_\rho \to \C^2$ of $(\C^2,0)$ and an irreducible component $E$ of the exceptional divisor $\rho^{-1}(0)$ of $\rho$, a \emph{test curve} at $E$ is any plane curve germ $(\gamma,0) \subset (\C^2,0)$ whose strict transform via $\rho$ is a smooth curve transverse to $E$ at a smooth point of $\rho^{-1}(0)$.
For the purpose of the criterion, it is sufficient to take for $\rho$ any sequence such that the strict transform $\Delta^*$  via $\rho$ of the discriminant curve $\Delta$ of  a generic plane projection  $\ell\colon (X,0)\to(\C^2,0)$ of $(X,0)$ 
 is a disjoint union of irreducible curves cutting the exceptional divisor $\rho^{-1}(0)$ of $\rho$ at smooth points (such as for example any good embedded resolution of $\Delta$), to consider a suitable subset $\{E_0,\ldots,E_s\}$ of the set of irreducible components of $\rho^{-1}(0)$, and to pick one test curve $\gamma_i$ at $E_i$ for each $i=0,\ldots,s$;
this gives rise to a set $\{\gamma_0,\ldots,\gamma_s\}$ called a \emph{family of test curves} for $(X,0)$ with respect to $\ell$.
We can now state the criterion.

\begin{theorem}[\cite{NeumannPedersenPichon2020a, FantiniPedersenPichonSchober2022}] \label{theorem:criterion_arcs_surfaces}
	Let $(X,0)$ be a normal surface singularity, let $\ell\colon (X,0)\to(\C^2,0)$ be a generic plane projection of $(X,0)$, and let  $\{\gamma_0,\ldots,\gamma_s\}$ be a family of test curves for $(X,0)$ with respect to $\ell$.
	Then the following conditions are equivalent:
	\begin{enumerate}
		\item $(X,0)$ is LNE.
		\item For every $j=0,\ldots,s$ and for every pair of distinct irreducible components $\xi$ and $\xi'$ of the principal part of $\ell^{-1}(\gamma_j)$, then $\xi$ and $\xi'$ have the same multiplicity as $\gamma_j$ and satisfy the equality
		\(
		q_i(\xi,\xi')=q_{o}(\xi,\xi').
		\)
	\end{enumerate}
\end{theorem}

In the statement, the \emph{principal part} of $\ell^{-1}(\gamma_j)$ is a curve obtained by deleting from $\ell^{-1}(\gamma_j)$ some irreducible components, namely those that do not pull back to curvettes on a suitable canonical subgraph of the minimal good resolution of $(X,0)$.
Contact orders between two complex curve germs $\xi$ and $\xi'$ are defined in a similar way as those between real arcs by looking at the shrinking rates as $\epsilon>0$ goes to $0$ of the inner or outer distance between the sets $\xi\cap\{||x||=\epsilon\}$ and  $\xi'\cap\{||x||=\epsilon\}$.
These are simple to compute in practice; for example this can be done by looking at the Puiseux expansions of the images of the irreducible curves $\xi$ and $\xi'$ through a second generic plane projection, which is easy to do with computer software such as Singular or Maple.

The version of the criterion of \cite{FantiniPedersenPichonSchober2022} improves upon that of \cite{NeumannPedersenPichon2020a} because the latter requires the map $\rho$ used to construct a test family to be a good embedded resolution of the family of the projections via $\ell$ of the polar curves with respect to \emph{all} generic plane projections of $(X,0)$ (which in particular demands to determine the Nash transform of the latter) and not just to one of them, then to consider a greater number of irreducible components of $\rho^{-1}(0)$, to take \emph{all} possible test curves at any such component, and finally to pull those curves back again with respect to \emph{all} generic plane projections of $(X,0)$.

\subsection{Examples}

The improvement of Birbrair and Mendes's arc criterion discussed in the previous subsection lead to the discovery of several infinite families of LNE complex surface germs with isolated singularities, no examples of which were previously known.

\begin{theorem}[\cite{NeumannPedersenPichon2020b}]
	\label{theorem_minimal_singularities_are_LNE}
	Let $(X,0)$ be a normal complex surface germ assume that it is rational.
	Then $(X,0)$ is LNE if and only if it is a minimal singularity.
\end{theorem}

This result gives the first known infinite family of non-conical LNE normal complex surface singularities, and is in fact the main reason why the criterion of \cite{NeumannPedersenPichon2020a} was developed.
For a thorough discussion of rational surface singularities we refer the reader to \cite{Laufer1972,Nemethi1999}, here we only recall that a surface singularity $(X,0)$ is \emph{rational} if and only if the exceptional divisor $E$ of its minimal good resolution consists of rational curves and its dual graph is a tree which satisfies a numerical condition (see \cite[Theorem~4.2]{Laufer1972}). 
If moreover $E$ is reduced, that is if the pullback of a general element of the maximal ideal of $(X,0)$ vanishes with order one along each component of $E$, then $(X,0)$ is said to be \emph{minimal}.

The fact that a rational singularity which is LNE is minimal based on Laufer's algorithm \cite[Proposition~4.1]{Laufer1972} to determine the fundamental cycle $Z_{\min}$ (see the footnote~\ref{footnote:fundamental_cycle} on page~\pageref{footnote:fundamental_cycle} for the definition of $Z_{\min}$).
Conversely, in order to apply the criterion of Theorem~\ref{theorem:criterion_arcs_surfaces} and show that a minimal surface singularity is LNE, the proof of Theorem~\ref{theorem_minimal_singularities_are_LNE} relies on a detailed study of the generic polar curves of minimal surface singularities performed in \cite{Spivakovsky1990}.

More generally, an equidimensional complex germ $(X,0)$ of multiplicity $m$ and embedding dimension $e$ is said to be \emph{minimal} if it is reduced, Cohen--Macaulay, with reduced tangent cone, and Abhyankar's inequality $m\geq e - \dim(X,0) +1$ is in fact an equality.
The last condition means that minimal singularities generally live in high-dimensional ambient spaces.
At the other side of the spectrum, the first family of LNE normal hypersurface singularities in $\C^3$ was discovered later by Misev and the first author of this survey, who studied LNE-ness among $\emph{superisolated singularities}$.
In order to define those, consider a complex hypersurface germ $(X,0)$ in $(\C^3,0)$, defined by the equation $f(x,y,z)=0$, and write $f$ as a sum of polynomials $f=f_d+f_{d+1}+\ldots$, with $f_d\neq 0$ and each $f_i$ homogeneous of degree $i$.
Then $(X,0)$ is said to be \emph{superisolated} if the plane projective curve defined by $f_{d+1}=0$ does not intersect the singular locus of the projectivized tangent cone $C_0X=\big\{(x:y:z)\in\mathbb P^2_{\mathbb C}\,\big|\,f_{d}(x:y:z)=0\big\}$ of $(X,0)$.
This implies that a single blowup of $X$ along $0$ is sufficient to resolve its singularities.

\begin{theorem}[\cite{MisevPichon2021}]
	Let $(X,0)$ be a superisolated normal complex surface germ.
	Then $(X,0)$ is LNE if and only if its tangent cone is reduced and LNE.
\end{theorem}

Recall that the tangent cone of a LNE singularity has to be reduced and LNE thanks to Theorem~\ref{thm:tangent cone}.

The further improvement obtained in \cite{FantiniPedersenPichonSchober2022} over the criterion of \cite{NeumannPedersenPichon2020a} allows to generalize the theorem above and obtain new families of LNE normal hypersurface singularities in $\C^3$.
In particular, the following result follows.

\begin{theorem}[\cite{FantiniPedersenPichonSchober2022}]
	\label{theorem_examples_LNE_from_refined_criterion}
Let $n$ and $k$ be two positive integers such that $n \geq k$ and let $(X,0)$ be the hypersurface in $(\C^3,0)$ defined by the equation
\[
\prod_{i=1}^{k} ( a_i x + b_i y ) - z^n  = 0,
\]
where the $(a_i,b_i)$'s are pairs of nonzero complex numbers such that the $k$ points $(a_i:b_i)$ of ${\mathbb P}^1_\C$ are pairwise distinct.
Then $(X,0)$ is LNE. 
\end{theorem}

Observe that whenever $n<k$ then the tangent cone of $(X,0)$ is defined by $z^n=0$. 
As the latter is non reduced, then $(X,0)$ cannot be LNE, or this would contradict Theorem~\ref{thm:tangent cone}.

\subsection{Properties of LNE surfaces}
\label{subsection_properties}

Lipschitz Normally Embedded complex surface singularities have many remarkable properties.
For example, the authors of this survey, together with Belotto da Silva, proved the following:

\begin{proposition}[{\cite[Proposition~2.2]{BelottodaSilvaFantiniPichon2020b}}]
\label{proposition:BdSFP_simple_properties_LNE_surfaces}
	Let $(X,0)$ be a complex LNE normal surface germ.
	Then the minimal resolution of $(X,0)$ factors through the blowup of $X$ along $0$, the exceptional components of this blowup are reduced, and the topological type of $(X,0)$ determines its multiplicity (which a priori is a datum of analytic nature).
\end{proposition}

The same paper also contains the following deeper result.


\begin{theorem}[{\cite[Theorems~1.1 and 1.2]{BelottodaSilvaFantiniPichon2020b}}]
	Let $(X,0)$ be a complex LNE normal surface germ.
	Then the topological type of $(X,0)$ determines the following data:
\begin{enumerate}
	\item \label{list_intro:property_polar} The  dual graph of the minimal good resolution of $(X,0)$ which factors through the blowup of the maximal ideal and through the Nash transform, decorated by two families of arrows corresponding respectively to the strict transform of a generic hyperplane section and to the strict transform of the polar curve of a generic plane projection.
	\item \label{list_intro:property_discriminant} The (embedded) topological type of the discriminant curve of a generic projection.
\end{enumerate}
	Moreover, this data can be computed explicitly from the dual graph of the minimal good resolution of $(X,0)$.
\end{theorem}

This theorem generalizes to all LNE normal surface germs results that were previously known only for minimal surface singularities.
In that special case, the first property was established by Spivakovsky  \cite[III, Theorem 5.4]{Spivakovsky1990}, while the second one was later proven by Bondil \cite[Theorem 4.1]{Bondil2003}, \cite[Proposition 5.4]{Bondil2016}.

This result can be thought of as a unique solution, for the class of LNE normal surface singularities, to the so-called problem of \emph{polar exploration}, which asks to determine the generic polar variety of a singular complex surface germ.
This problem was studied for a general surface germ $(X,0)$ by the same authors together with Némethi \cite{BelottodaSilvaFantiniNemethiPichon2021}, relying on the study of the \emph{inner rates} of $(X,0)$.
Those are an infinite family of metric invariants that appeared naturally in the study of the Lipschitz geometry of $(X,0)$ in the foundational work \cite{BirbrairNeumannPichon2014}, and were then systematically studied in \cite{BelottodaSilvaFantiniPichon2019}.
From this point of view, it is worth noticing that in the paper \cite{BelottodaSilvaFantiniPichon2020b} referred to above it is also shown that the topological type of an LNE normal surface germ $(X,0)$ determines its inner rates (see Proposition~5.1 of \emph{loc. cit.}), and this combined with the main result of \cite{BelottodaSilvaFantiniPichon2019} is what allows them to determine the combinatorics of the polar curve of a generic plane projection of $(X,0)$.


\section{Open questions}

We conclude this survey by putting forward some open questions about LNE singularities that we find worth of interest.

\subsection{Behavior under blowup and Nash transform}

It was a long-held belief by several experts in the field that the point blowup of a LNE complex surface germ $(X,0)$ with an isolated singularity would most likely have itself only LNE singularities.
The following counterexample took therefore the authors by surprise.

\begin{example}
	The hypersurface $(X,0)$ in $(\C^3,0)$ defined by the equation
\[
	(x+y)(2x+y)(x+2y)-z^5 = 0
\]
	is LNE (it is a special case of Theorem~\ref{theorem_examples_LNE_from_refined_criterion} from \cite{FantiniPedersenPichonSchober2022}).
	However, the blowup of $X$ along $0$ has a singularity whose local equation is  $2 x_1^3 + 7x_1^2y_1 + 7x_1y_1^2 + 2y_1^3 - w^2=0$, whose tangent cone $w^2=0$ is non reduced. 
\end{example}	

On the other hand, the following  question is still open.

\begin{quest}\label{question_Nash_transform_is_LNE}
	Let $(X,0)$ be a LNE complex surface germ with an isolated singularity.
	Does the Nash transform of $(X,0)$ have itself only LNE singularities?
\end{quest}

Observe that by \cite[Corollary~4.7]{BelottodaSilvaFantiniPichon2020b}, if $(X,0)$ is LNE then its Nash transform has only \emph{sandwiched singularities}.
Since sandwiched singularities are rational, in order to give a positive answer to Question~\ref{question_Nash_transform_is_LNE} thanks to Theorem~\ref{theorem_minimal_singularities_are_LNE} one would have to show that they are minimal.

\subsection{Topological types of LNE surface singularities}

It is a very natural question to study the topological properties of LNE singularities.
In order to start such an investigation, it seems wise to restrict oneself to the case of normal complex surfaces.
In this context, it is well-known that, by a classical result of Neumann, the topological type of a normal surface singularity $(X,0)$ is equivalent to the datum of the weighted dual graph $\Gamma_\pi$ of the minimal good resolution $\pi$ of $(X,0)$, where each vertex is weighted by the genus and self-intersection of the corresponding exceptional component of $\pi$.

A first observation is then that being LNE is not a topological property, as shown by the following example, kindly provided to us by Jan Stevens.

\begin{example}
	\label{example:jan_stevens}
Let $X_1$ be the hypersurface in $\mathbb C^3$ defined by the equation $x^4+y^4+z^4=0$ and let $X_2$ be the surface in $\mathbb C^4$ defined by the equations $y^2=xz$ and $w^2=x^4+z^4$.
The two surface germs $(X_1,0)$ and $(X_2,0)$ are normal and have the same topological type, since for both of them the exceptional divisor of the minimal resolution is a single curve of genus $3$ and self-intersection $-4$.	
However, $(X_1,0)$ is LNE, since it is the cone over the smooth projective curve $x^4+y^4+z^4=0$, while $(X_2,0)$ is not, since its tangent cone, which is defined by the equations $w^2=0$ and $y^2=xz$, is non reduced.
\end{example}

However, given a weighted graph $\Gamma$ one can say that $\Gamma$ is \emph{LNE} if there exists a LNE normal surface singularity with resolution graph $\Gamma$.
The following question is therefore very natural.

\begin{quest}
		Is there a combinatorial characterization of LNE weighted graphs?
\end{quest}

The first results of \cite{BelottodaSilvaFantiniPichon2020b} provide some obstructions for a weighted graph $\Gamma$ to be LNE.
Denote by $Z_{\min}$ the fundamental cycle of $\Gamma$.
We then say that a vertex $v$ of $\Gamma(V)$ is a \emph{numerical $\mathcal L$-node} of $\Gamma$ if $E_v\cdot Z_{\min}<0$.
\footnote{\label{footnote:fundamental_cycle}
	Let us briefly recall the precise definitions of the combinatorial notions we use here, in particular that of the fundamental cycle and how to make sense of the intersection number $E_v\cdot Z_{\min}$. 
	A \emph{weighed graph} is a finite connected graph $\Gamma$ without loops and such that each vertex $v \in V(\Gamma)$ of $\Gamma$ is weighted by two integers, its \emph{genus} $g(v)\in\Z_{\geq 0}$ and its \emph{self-intersection} $e(v)\in \Z_{\leq 0}$.	
	Let $E=\bigcup_{v\in V(\Gamma)}E_v$ be a configuration of curves whose dual weighted graph is $\Gamma$, so that in particular $g(v)=g(E_v)$ and $E_v^{\;2}=e(v)$, and let $I_{\Gamma} = (E_v \cdot E_{v'})$ be the \emph{incidence matrix} of $\Gamma$.
	We assume that $I_{\Gamma}$ is negative definite.
	A \emph{divisor on $\Gamma$} is a formal sum $D =\sum_{v \in V(\Gamma)} m_v E_v$ over the set of the irreducible components of $E$ with integral coefficients.
	The \emph{fundamental cycle} $Z_{\min}$ of $\Gamma$ is then the unique nonzero divisor on $\Gamma$ which is minimal among those divisors $D$ satisfying $D\cdot E_v<0$ for all $v\in V(\Gamma)$.
	Its existence was shown by Artin, and its coefficients are all strictly positive.
}
It was then shown in \cite[Proposition~2.2.(ii)]{BelottodaSilvaFantiniPichon2020b} that whenever $(X,0)$ is a LNE singularity whose weighted dual graph is $\Gamma$ then $Z_{\min}$ coincides with the maximal ideal cycle $Z_{\max}$ of $(X,0)$.
\footnote{The cycle $Z_{\max}$ is the divisor on $\Gamma$ whose coefficient at a vertex $v$ is the order of vanishing of a generic linear form of $(X,0)$ along the exceptional component $E_v$ associated with $v$.
}
In particular, the numerical $\mathcal L$-nodes of $\Gamma$ coincide with its usual $\mathcal L$-nodes, which are the vertices which correspond to the exceptional components of the blowup of $X$ along $0$.
It then follows from Proposition~\ref{proposition:BdSFP_simple_properties_LNE_surfaces} that the numerical $\mathcal L$-nodes of $\Gamma$ are reduced, which means that whenever $\Gamma$ is the dual resolution graph of a LNE surface then it satisfies the following combinatorial condition:
\[
 \text{ writing } Z_{\min}=\textstyle\sum d_v E_v, \text{ we have } d_v = 1  \text{ for every } v \text{ such that }Z_{\min}\cdot E_v <0. 
\]

A weighted graph satisfying the condition above is called a \emph{Kodaira graph}.
Kodaira graphs are precisely those which can be realized as dual resolution graphs of the so-called \emph{Kodaira singularities}, a class of surface singularities defined in terms of a suitable family of curves and introduced by Karras \cite{Karras1980} after work of Kulikov \cite{Kulikov1975}.
It seems worthwhile of interest to fully investigate the relations between Lipschitz Normal Embeddings and Kodaira singularities (or the subclass of Kodaira singularities consisting of the so-called \emph{Kulikov singularities} introduced by Stevens \cite{Stevens2018}).
As a first step towards this, we mention that among rational singularities the only ones that are Kodaira are precisely the minimal singularities (see \cite[Example~2.8 plus Theorem~2.9]{Karras1980}), that is the ones that are also LNE (and Kulikov).
However, not all Kulikov singularities (and therefore not all Kodaira singularities) are LNE, since their projective tangent cone is not necessarily reduced (see \cite[Example~2.4]{Stevens2018} for a Kulikov singularity with reducible tangent cone; moreover its minimal resolution does not factor through the blowup of its maximal ideal).

\subsection{Generalizations of the arc criterion}

We have mentioned in Section~\ref{sec:surfaces} how improving the arc criterion for LNE-ness of Theorem~\ref{theorem:criterion_arcs_original} proved to be extremely useful in the study of LNE complex surface germs.
It would therefore be very interesting to find similar improvements in a more general setting, and in particular for complex germs of higher dimensions.

\begin{quest}
	Find an improvement of the arc criterion of Birbrair and Mendes (Theorem~\ref{theorem:criterion_arcs_original}) that only requires to compare the inner and outer contact orders of a finite (and in fact as small as possible) family of pairs of real arcs, for complex germs of arbitrary dimension.
\end{quest}	
	
Let now $(X,0)$ be an algebraic complex germ.
The \emph{arc space} $\mathcal L_\infty(X,0)$ of $(X,0)$ is a scheme that parametrizes all complex arcs on $X$ that are centered in $0$, which are by definition the points of $X$ with coordinates in $\C[[t]]$ and such that setting $t=0$ we obtain the complex point $0\in X$.
Its geometry, and the geometry of the \emph{jet spaces} of $(X,0)$ (the varieties parametrizing the \emph{jets} of $(X,0)$, which are its complex arcs truncated at a given order), reflect interesting properties of the singularity of $(X,0)$.
Their study plays an important role in many subareas of algebraic geometry, such as in the theory of motivic integration.
This leads us to formulate the following problem.

\begin{quest}
	Give a criterion for the LNE-ness of a germ $(X,0)$ in terms of the geometries of the arc or jet spaces of $(X,0)$.
\end{quest}

Such a criterion should involve testing the contact orders for generic arcs (or families of arcs, which are commonly called \emph{wedges}) of some suitable irreducible subschemes of the arc space.
In dimension 2, this could be related to the irreducible components of $\mathcal L_\infty(X,0)$, and hence to the \emph{essential valuations} of $(X,0)$, thus relating Lipschitz geometry to the notorious Nash problem solved by Fern\'andez de Bobadilla and Pe Pereira in \cite{BobadillaPereira2012}.
In arbitrary dimension, the LNE-ness of a germ $(X,0)$ could possibly be read in terms of its \emph{terminal valuations}, whose relation to the geometry of $\mathcal L_\infty(X,0)$ was detailed by de Fernex and Docampo \cite{FernexDocampo2016}.
	
More generally, the relations between the geometry of arc and jet spaces and Lipschitz geometry are completely unexplored. 
Some recent results, such as the appearance of Mather discrepancies in \cite{BelottodaSilvaFantiniNemethiPichon2021}, suggest that this may be a matter worth exploring.

\subsection{Higher dimensional LNE complex singularities}

As should now be clear to the reader, very little is known about complex LNE singularities starting from the dimension 3.
Since the simplest family of LNE surface germs consists of minimal surface singularities, the following question is very natural.

\begin{quest}
	Are minimal singularities in arbitrary dimension LNE?
\end{quest}

Recall that the definition of minimal singularities in arbitrary dimensions appears after Theorem~\ref{theorem_minimal_singularities_are_LNE}.

Minimal singularities form a building block in the Minimal Model Program.
Therefore, more generally, can we hope to characterize LNE singularities, or at least provide new classes of higher-dimensional examples, using the invariants appearing in the Minimal Model Program?

\addtocontents{toc}{\SkipTocEntry}
\subsection*{Acknowledgements}

We are very grateful to José Seade, who proposed to us to write a contribution for this volume, and to Jan Stevens, who showed us the example included in~\ref{example:jan_stevens}.
We would also like to thank all the friends and colleagues with whom we had many inspiring discussions about Lipschitz geometry, and in particular Helge Pedersen and Bernd Schober who graciously agreed to let us discuss our joint work in progress in this survey.
This work has been partially supported by the project \emph{Lipschitz geometry of singularities (LISA)} of the \emph{Agence Nationale de la Recherche} (project ANR-17-CE40-0023).
It was initiated at the \emph{Centre International de Rencontres Mathématiques} in Luminy during the \emph{Jean-Morlet Chair} held in 2021 by Javier Fern\'andez de Bobadilla; we thank the CIRM for the hospitality.

\bibliographystyle{alpha}
\bibliography{bibliography}

\vfill

\end{document}